\documentclass[a4paper,reqno]{amsart}
\numberwithin{equation}{section}

\usepackage[dvips]{graphicx}
\usepackage{rotating}
\usepackage{graphics}
\usepackage{epsfig}
\usepackage{amsmath}
\usepackage{amssymb}
\usepackage{amsthm}
\usepackage{amsfonts}
\usepackage{latexsym}
\usepackage{psfrag, float}
\usepackage{multirow}
\usepackage{bezier}
\usepackage{curves}

\def\?{(?)\marginpar{|?}}

\newtheorem{theorem}{Theorem}[section]

\newtheorem{lemma}[theorem]{Lemma}

\newtheorem{corollary}[theorem]{Corollary}

\def\beq{\begin{equation}}
\def\eeq{\end{equation}}
\def\be{\begin{equation*}}
\def\ee{\end{equation*}}
\begin{document}

\title{ON OPTIMAL REPRESENTATIVES OF FINITE COLOURED LINEAR ORDERS}

\author{F. Mwesigye and J.K. Truss}
  \address{Department of Pure Mathematics,
          University of Leeds,
          Leeds LS2 9JT, UK}
  \email{pmtjkt@leeds.ac.uk, feresiano@yahoo.com}

\keywords{coloured linear order, Ehrenfeucht-Fra\"iss\'e game, optimality, classification\\
Supported by a London Mathematical Society scheme 5 grant}


\begin{abstract}
Two structures $A$ and $B$ are $n$-equivalent if player II has a winning strategy in the $n$-move Ehrenfeucht-Fra\"iss\'e game 
on $A$ and $B$. We extend earlier results about $n$-equivalence classes for finite coloured linear orders, describing an 
algorithm for reducing to canonical form under 2-equivalence, and concentrating on the cases of 2 and 3 moves.
\end{abstract}

\maketitle
%\tableofcontents

%%%%%%%%%%%%%%%%%%%%%%%%%%%%%%%%%%%%%%%%%%%%%%%%%%%%%%%%%%%%%%%%%%%%%%%%%%%%%%%%%%%%%
%%%%%%%%%%%%%%%%%%%%%%%%%%%%%%%%%%%%%%%%%%%%%%%%%%%%%%%%%%%%%%%%%%%%%%%%%%%%%%%%%%%%%

\section{\textbf{Introduction}}\label{intro}

In \cite{mwesigye2} we studied the equivalence of finite coloured linear orders up to level $n$ in an Ehrenfeucht-Fra\"iss\'e game, written as $\equiv_n$, which means that player II has a winning strategy 
in this game, as well as making some remarks about the infinite case. We gave some bounds for the minimal representatives in the finite case, and the infinite case for up to 2 moves. These results were 
extended in \cite{mwesigye3} to all coloured ordinals, in the monochromatic case giving a precise list of optimal representatives, and in the coloured case giving bounds.

In this paper we return to the finite case, and extend the work of the first paper, by improving the bounds in some instances, and throwing further light on uniqueness of representatives. First we 
briefly recall the required definitions. A {\em coloured linear ordering} is a triple $(A, <, F)$ where $(A, <)$ is a linear order and $F$ is a mapping from $A$ onto a set $C$ of `colours'. We just write 
$A$ instead of $(A, <, F)$ provided that the ordering and colouring are clear. In the $n$-move Ehrenfeucht-Fra\"iss\'e game on coloured linear orders $A$ and $B$ players I and II play alternately, I 
moving first. On each move I picks an element of either structure, and II responds by choosing an element of the other structure. After $n$ moves, I and II between them have chosen elements 
$a_1, a_2, \ldots, a_n$ of $A$, and $b_1, b_2, \ldots, b_n$ of $B$, and player II {\em wins} if the map taking $a_i$ to $b_i$ for each $i$ is an order and colour-isomorphism, and player I wins 
otherwise. We say that $A$ and $B$ are {\em $n$-equivalent} and write $A \equiv_n B$, if II has a winning strategy. Then $\equiv_n$ is an equivalence relation having just finitely many $n$-equivalence 
classes. An {\em optimal representative} is a member of an equivalence class of least possible length. 

In \cite{mwesigye2} we gave upper bounds for the lengths of optimal representatives of $\equiv_n$-classes of $m$-coloured finite linear orders. Only for the case $n = 2$ were these bounds exact. We return 
to this case, describing explicitly the classification of finite $m$-coloured linear orders up to $\equiv_2$-equivalence (based on the idea of a `$T$-configuration' introduced in \cite{mwesigye2}). From 
this we are able to read off which equivalence classes are finite or infinite, and provide an algorithm for determining an optimal representative corresponding to any given finite $m$-coloured linear order. 
We also show that a finite coloured linear order is $\equiv_2$-optimal if and only if each 1-character (see below for the definition) appears at most once.

The problem for more than 2 moves seems to be quite hard, so we concentrate on the case of 3 moves. The idea is that using a key inductive lemma from \cite{mwesigye2}, we need to understand better how the 
2-characters behave, and that is the reason for re-examining the case $n = 2$ in more detail. 

Next we recall the notion of `character' from \cite{mwesigye2}, and the main result about characters. Assume that we have found representatives for the $n$-equivalence classes of certain $m$-coloured 
linearly ordered sets. We write the representative for $A$ as $[A]_n$. In a coloured linear order $A$, the \textit{$n$-character} of $a \in A$ having colour $c$ is the ordered pair 
$\langle [A^{<a}]_n, [A^{>a}]_n \rangle$ (where $A^{< a} = \{x \in A: x < a\}$ and $A^{> a} = \{x \in A: x > a\}$). We let $\rho_n^c(A) = \{\langle[A^{<a}]_n,  [A^{>a}]_n \rangle : a \in A$ is 
$c$-coloured$\}$. Here we shall always include the colour as part of the $n$-character of $a$, in which case we write it as $\langle [A^{<a}]_n, [A^{>a}]_n \rangle_c$ (formally this would be an ordered 
triple).

\begin{theorem} [\cite{mwesigye2}] $A \equiv_{n+1}B$ if and only if $\rho_n^c(A) = \rho_n^c(B)$ for all $c \in C$. \label{1.1}
\end{theorem}

We need the following `Cutting lemma' from \cite{mwesigye2}. 

\begin{lemma} \label{1.2} Let $A$ be a finite $m$-coloured linear order and let $a$ and $b$ be elements of $A$ such that $a < b$ satisfying the following conditions:
\begin{itemize}
\item[(i)] $a$ and $b$ determine the same $n$-character,
\item[(ii)] for every $x \in A$ with $a < x \le b$,  there is $y \le a$ having the same $n$-character as $x$.
\end{itemize}
Then $A$ is $(n+1)$-equivalent to $B = A \setminus (a, b]$.
\end{lemma}

\vspace{.1in}

Note that this may be applied in a trivial case, namely, that no two consecutive points of an $\equiv_{n+1}$-optimal finite string can have equal $n$-characters.

\vspace{.1in}

It is clear from Theorem \ref{1.1} that if in an $m$-coloured linear order, no two points have the same $n$-character, then the ordering is $\equiv_{n+1}$-optimal, meaning that it is not $(n+1)$-equivalent 
to any shorter ordering. Based on this, we present a construction of a finite 2-coloured linear order of length 70 in which all points have distinct 2-characters, and which is therefore $\equiv_3$-optimal,
and show that 70 is the greatest possible number in which all 2-characters are distinct. We also construct a finite coloured $\equiv_3$-optimal linear order of length 74, in which 2-characters must 
therefore repeat. It should possible to find longer examples, but the details would be quite tedious, so giving one of this length is good enough to illustrate the idea. This casts some light on 
the hypothesis required for the `cutting lemma' (that is, what it says isn't that we can reduce the length just based on repetition of characters---more is required about what happens in between). 

The typical case we have in mind is that when searching for optimal representatives, we start with a possibly long coloured order, and successively reduce it by removing pieces, retaining $n$-equivalence, 
till it becomes optimal. The proof of \cite{mwesigye2} is too indirect to guarantee immediately that the final ordering is a subordering of the one we start with. We therefore extend the material of 
\cite{mwesigye2} by showing that for 2-equivalence at any rate, we {\em can} guarantee that the optimal representative is contained in the original one; we present an algorithm for achieving this. We 
believe that this is false for $n = 3$, and in section 3 explain why.

With regard to the general case, but particularly applied to $n = 3$, we use directed graphs to help analyze $n$-equivalence. One method would be to take $(n-1)$-characters themselves as vertices of the 
directed graph, with an arrow going from $\langle X_1, Y_1 \rangle_{c_1}$ to $\langle X_2, Y_2 \rangle_{c_2}$ if for some representatives $x_1, x_2, y_1, y_2$ of $X_1, Y_1, X_2, Y_2$, 
$x_1c_1 \equiv_{n-1} x_2$ and $c_2y_2 \equiv_{n-1} y_1$, where these are the strings obtained from $x_1$, $y_2$ by adding a $c_1$-coloured point on the right, a $c_2$-coloured point on the left respectively. 
The idea is that in scanning a (long) word from left to right, at each point we can view its $(n-1)$-character to left and right, and see how this varies. In practice in what we present here for $m = 2$, 
$n = 3$, we focus just on the `middle' section of the given string, in which case a simplified directed graph gives all the information we require. 

\section{Classification of 2-equivalence classes}

In this section we give a lot more detail about the 2-equivalence classes of finite coloured linear orders. In \cite{mwesigye2} we established the precise value ($m^2 + 2m$) of the least upper bound of the 
lengths of the optimal representatives of $\equiv_2$-classes. Here we are able, using similar ideas, to give an explicit list of all the $\equiv_2$-classes, from which we can read off, for instance, the 
length of the optimal representative of each class, and also note which classes are finite or infinite. The key idea here is to use the notion of `$T$-configuration' which was introduced in \cite{mwesigye2}. 

We fix $m$ as the number of colours. A {\em $T$-configuration} is then defined to be a linear order of the form $T = \{x_i: 1 \le i \le m\} \cup \{y_i: 1 \le i \le m\}$ in which $x_1 < x_2 < \ldots < x_m$ 
and $y_1 > y_2 > \ldots > y_m$, and $x_1$ and $y_1$ are the least and greatest members of $T$ respectively. Here all the $x_i$ are therefore distinct, and so are the $y_i$, but it is not ruled out that 
$x_i = y_j$ for certain $i$ and $j$. Each $T$-configuration therefore has size between $m$ and $2m$. If $(A, \le, F)$ is a finite coloured linear order having $m$ colours, then there is an associated 
$T$-configuration, which is the linear order induced on $\{x_i: 1 \le i \le m\} \cup \{y_i: 1 \le i \le m\}$ where for each $i$, $x_i$ is the least point $x$ of $A$ such that $\{F(z): z \in A, z \le x\}$ 
has $i$ elements, and $y_i$ is the greatest point $y$ of $A$ such that $\{F(z): z \in A, z \ge y\}$ has $i$ elements. Under these circumstances, the $T$-configuration becomes coloured. However, the same 
$T$-configuration may be coloured in several different ways. We remark that not all $T$-configurations are associated with a coloured linear order. The following lemma explains when this happens.

\begin{lemma} The $T$-configuration $T = \{x_i: 1 \le i \le m\} \cup \{y_i: 1 \le i \le m\}$ is associated with some $m$-coloured linear order if and only if $i + j \le m+1 \Rightarrow x_i \le y_j$.   
\label{2.1}   \end{lemma}

\begin{proof} First to check the necessity of the given condition, suppose that $\{x_i: 1 \le i \le m\} \cup \{y_i: 1 \le i \le m\}$ is the $T$-configuration arising from the coloured linear order
$(A, \le, F)$, and let $i + j \le m+1$. Let $k$ be greatest such that $x_i \le y_k$. Then $\{y_l: l > k\}$ are distinctly coloured points lying in $(-\infty, x_i)$, which exhibits $i-1$ colours. Hence
$m-k \le i-1$, so $m+1 \le i+k$. We deduce that $i+j \le i+k$ and hence $j \le k$, so that $x_i \le y_j$.

Conversely, assuming the given condition holds, let the $T$-configuration $T = \{x_i: 1 \le i \le m\} \cup \{y_i: 1 \le i \le m\}$ be given, and we have to find a coloured linear order $(A, \le, F)$ such 
that $T$ is the associated $T$-configuration. We take $A = T$, and have to show how the points can be coloured so that $x_i$ is the least point $x$ of $A$ such that $\{F(z): z \in A, z \le x\}$ has $i$ 
elements, and $y_i$ is the greatest point $y$ of $A$ such that $\{F(z): z \in A, z \ge y\}$ has $i$ elements. Let us start by colouring the $x_i$ by distinct colours. Clearly this ensures that $x_i$ is the 
least point such that $\{F(x_k): x_k \in A, x_k \le x_i\}$ has $i$ elements. We have to colour the $y_j$ so that no member of $\{F(z): z \in A, z \le x_i\}$ has a `new' colour. We assign colours 
successively to $y_m$, $y_{m-1}$, $\ldots$, $y_1$ according to which of the sets $\{x_1\}$, $(x_1, x_2)$, $\{x_2\}$, $(x_2, x_3)$, $\ldots$, $\{x_m\}$, $(x_m, \infty)$ they lie in. Given $i$, let $j$ be the 
least such that $y_j < x_{i+1}$ (if any). Then by hypothesis, $i + 1 + j \not \le m+1$, so $j > m-i$. Hence there are at most $i$ values of $j$ such that $y_j < x_{i+1}$.

The colouring is now given as follows. If $y_j = x_i$ then we let $F(y_j) = F(x_i)$. Otherwise consider colouring all the $y_j$s which lie in $(x_i, x_{i+1})$. By the remark just made, there are at most $i$ 
values of $j$ such that $y_j < x_{i+1}$, and there are $i$ colours available for $\{y_j: y_j < x_{i+1}\}$. We have so far used $|\{j: y_j \le x_i\}|$ of these colours, so the number remaining is 
$i - |\{j: y_j \le x_i\}| \ge |\{j: y_j < x_{i+1}\}| - |\{j: y_j \le x_i\}|  = |\{j: x_i < y_j < x_{i+1}\}|$, and these points are coloured in any way using the available colours. 

The construction has explicitly ensured that for each $i$, $x_i$ is the least point such that $(-\infty, x_i]$ is coloured by $i$ colours. To verify the corresponding condition for $y_i$, note that
there are certainly exactly $i$ values of $k \le i$ such that $y_k \ge y_i$, and these points are all coloured by distinct colours. Suppose that $x_j \ge y_i$. Then as there are $m$ colours, and all $y_k$ 
points are distinctly coloured, there is $k$ such that $F(x_j) = F(y_k)$. If $k > i$ then $y_k < y_i \le x_j$, contrary to $x_j$ the {\em least} point coloured $F(x_j)$. We deduce that $k \le i$, and so
$F(x_j) \in \{F(y_k): k \le i\}$ as required. \end{proof}

\vspace{.1in}

To specify a finite coloured order up to 2-equivalence, we need to know in addition what colouring $T$ receives (and then call this a {\em coloured $T$-configuration}, though we do not indicate this 
explicitly in the notation), and which colours arise as the colours of points lying between any two consecutive members of $T$, and we write the set of colours between $u$ and $v$ as $g(u, v)$. We write 
${\mathcal C}_{T,g}$ for the set of all finite coloured orders such that $T$ is the associated coloured $T$-configuration and colours between the points are given by $g$. Note that not all possible sets of 
colours are possible for $g(u, v)$ and they will be constrained by the $x_i$ and $y_j$. If for ease we write $x_{m+1} = \infty$ and $y_{m+1} = -\infty$ (not coloured) then a point with colour $c$ can be 
inserted in $(x_i, x_{i + 1}) \cap (y_{j+1}, y_j)$ if and only if $c \in F(-\infty, x_i] \cap F[y_j, \infty)$.

\begin{theorem} Two finite $C$-coloured linear orders are $2$-equivalent if and only if for some $T$ and $g$, where $T$ is a coloured $T$-configuration, they both lie in ${\mathcal C}_{T,g}$. 
In other words, the ${\mathcal C}_{T,g}$ classify the $\equiv_2$-classes of $m$-coloured finite linear orders.    \label{2.2}   \end{theorem}

\begin{proof} This relies on Theorem \ref{1.1}, which tells us that $A \equiv_2 A'$ if and only if they exhibit the same 1-characters. Let $\{x_i: 1 \le i \le m\} \cup \{y_i: 1 \le i \le m\}$ and
$\{x'_i: 1 \le i \le m\} \cup \{y'_i: 1 \le i \le m\}$ be the coloured $T$-configurations associated with $A$ and $A'$, and suppose first that $A \equiv_2 A'$. Thus $A$ and $A'$ exhibit the same 
1-characters. Now by definition of $x_i$, $|\{F(z): z < x_i\}| = i - 1$ and $|\{F(z): z \le x_i\}| = i$. Let $x' \in A'$ realize the same 1-character in $A'$ as $x_i$ does in $A$. Then $F'(x') = F(x_i)$, 
$|\{F'(z): z < x'\}| = i - 1$ and $|\{F'(z): z \le x'\}| = i$. This implies that $x' = x_i'$. Similarly for $y_j$ and $y_j'$. Next we have to see that $x_i \le y_j \Leftrightarrow x_i' \le y_j'$, and 
similarly for $<$. Suppose that $x_i \le y_j$ ($x_i < y_j$ respectively). Then $x_i$ has at least $j-1$ colours to the right (at least $j$ respectively), and as it realizes the same character as $x_i'$, this 
is also true of $x_i'$ in $A'$, and so it follows that $x_i' \le y_j'$ ($x_i' < y_j'$ respectively). We deduce that $A$ and $A'$ realize the same coloured $T$-configurations. To see that they realize the 
same functions $g$, let $u < v$ be consecutive members of $\{x_i: 1 \le i \le m\} \cup \{y_i: 1 \le i \le m\}$, and $u' < v'$ the corresponding consecutive members of 
$\{x'_i: 1 \le i \le m\} \cup \{y'_i: 1 \le i \le m\}$. Then the left and right 1-characters of each member of $(u, v)$ are $\{F(z): z \le u\}$ and $\{F(z): z \ge v\}$ respectively, and furthermore, these 
characters are not realized by any other members of $A$. Precisely these same left and right characters are realized in $(u', v')$, and since the only extra ingredient required to specify the character is 
the colour of the point, it follows that exactly same set of colours is realized in $(u, v)$ and $(u', v')$. In other words, $g(u, v) = g'(u', v')$.

Conversely, supposing that $A$ and $A'$ both lie in the same ${\mathcal C}_{T,g}$, we see that they both realize the same 1-characters, so are 2-equivalent.  \end{proof}

\begin{corollary} A $\equiv_2$-class of finite linear coloured orders is finite if and only if it is a singleton, which holds if and only if $g(u,v) = \emptyset$ for each $u, v$.   \label{2.3}
\end{corollary}

Next we give an algorithm for determining an optimal member of the 2-equivalence class of a finite coloured linear order $A$. It would be possible to do this inductively on the number of colours, and since 
we shall require them later anyway, we define the subsets $L$, $R$, and $M$ of $A$, for `left', `right', and `middle'. A point lies in $L$ if there are fewer than $m$ colours occurring to its left, and is 
in $R$ if there are fewer than $m$ colours occurring to its right. The remainder (if any) is $M$. Thus in the above notation, $L = (-\infty, x_m)$ and $R = (y_m, \infty)$. The induction would be based on 
the fact that each of $L$ and $R$ exhibit only $m-1$ colours. There are some (minor) complications in the case where $L$ and $R$ overlap however, so the following method, based on Theorem \ref{2.2}, is 
preferable.

From $A$ we first evaluate $x_i$ and $y_j$. Then we replace each interval $(u,v)$ by one in which each of its colours only arises once. This leads to the following result.

\begin{theorem} Any finite coloured string has a $2$-equivalent $\equiv_2$-optimal substring.  \label{2.4} \end{theorem}

Let us see how the algorithm works out for small values of $m$:

If $m = 0$ there is only one order, namely $\emptyset$.

If $m = 1$ with colour $r$, then there are two possible $T$-configurations, with $x_1 = y_1$ or $x_1 < y_1$. The former gives us just a singleton $r$ (since there is no interval of consecutive points into 
which new elements can be inserted), and the latter $rr$ which is a singleton $\equiv_2$-class, and $rrr$, which lies in the infinite $\equiv_2$-class $\{r^n: 3 \le n\}$ (where $g(x_1, y_1) = \{r\}$).

If $m = 2$, these are the possible $T$-configurations, with the corresponding singleton $\equiv_2$-classes given:

$x_1 < x_2 < y_2 < y_1$, $rbrb$, $rbbr$, $brbr$, $brrb$,

$x_1 < x_2 = y_2 < y_1$, $rbr$, $brb$,

$x_1 < y_2 < x_2 < y_1$, $rrbb$, $bbrr$,

$x_1 = y_2 < x_2 < y_1$, $rbb$, $brr$,

$x_1 < y_2 < x_2 = y_1$, $rrb$, $bbr$,

$x_1 = y_2 < x_2 = y_1$, $rb$, $br$.

Including the allowed insertions, where we write $r^k$ for an arbitrary sequence of $k$ $r$s ($k \ge 0$), similarly $b^l$, and $w(r,b)$ an arbitrary string of $r$s and $b$s, this gives rise to the following 
list

for $rbrb$: $rr^kbw(r,b)rb^l$, 16 $\equiv_2$-classes (two options for each of $k$ and $l$, and four for $w(r,b)$), similarly for $brbr$, $rbbr$ and $brrb$,

for $rbr$: $rr^kbr^lr$, 4 $\equiv_2$-classes, similarly for $brb$,

for $rrbb$: $rr^krbb^lb$, 4 $\equiv_2$-classes, similarly for $bbrr$,

for $rbb$: $rbb^lb$, 2 $\equiv_2$-classes, similarly for $brr$, 

$rrb$, $bbr$ are similar to $rbb$,

$rb$, just one $\equiv_2$-class, and $br$ is similar.

This gives a total of 90 $\equiv_2$-classes in which two colours appear. Note that the optimal representative of each class is unique, except when there is a `middle' section in which both colours appear. 
For instance, $rrbrbrbb \equiv_2 rrbbrrbb$, though each is of optimal length.

If $m = 3$, there are 26 possible $T$-configurations, of which all but four fulfil the stipulations of Lemma \ref{2.1} (the four which do not are given by $x_1 \le y_3 < y_2 < x_2 < x_3 \le y_1$). To list 
even these is quite laborious, and when their possible colourings are taken into account, as well as the possible insertions, it is seen that the list increases dramatically over the case $m = 2$. For 
instance, for the $T$-configuration $x_1 < x_2 < x_3 < y_3 < y_2 < y_1$ there are 36 ways of colouring the points, and for the $rbg$ colouring of $x_i$ and $y_i$ points, the $\equiv_2$-classes are of the 
forms $rr^{i_1}br^{i_2}b^{j_2}gr^{i_3}b^{j_3}g^{k_3}rb^{j_4}g^{k_4}bg^{k_5}g$ where the indices are all 0 or 1, giving $2^9$ possibilities, so even for this case there are $36 \times 2^9 = 18432$ 
$\equiv_2$-classes.

We remark that the easiest way to demonstrate that a finite string is optimal in its $\equiv_2$-class is to show that all its points have distinct 1-characters (then appeal to Theorem \ref{1.1}), and in fact 
this suffices for all 90 strings for $m = n = 2$, as one sees by inspection. The same holds for any number of colours (though not with greater values of $n$, as we see in the next section).

\begin{theorem} For any $m$, no $\equiv_2$-optimal $m$-coloured string realizes the same $1$-character more than once.  \label{2.5}
\end{theorem}

\begin{proof} Suppose on the contrary that $(A, <, F)$ is $\equiv_2$-optimal but $a < b$ realize the same 1-character (and have the same colour). We show that $A \equiv_2 A \setminus \{b\}$, contradicting 
optimality of $A$. This is similar to the Cutting Lemma, Lemma \ref{1.2}. We just need to show that $A$ and $A' = A \setminus \{b\}$ realize the same 1-characters (with colours). For this 
we note that if $x \neq b$, then $x$ realizes the same 1-character in $A$ and $A'$, and if $x = b$ then $x$ realizes the same 1-character in $A$ as $a$ does in $A'$. In each case, the colours of $x$ and its 
replacement are equal, so as 1-characters are entirely determined by the sets of colours occurring on left and right, we just need to look at the colours occurring in $A^{< x}$, $A'^{< x}$, $A^{> x}$, 
$A'^{> x}$, and in the second case, $A^{< b}$, $A'^{< a}$, $A^{> b}$, $A'^{> a}$.

If $x < b$ then $A^{< x} = A'^{< x}$, and the colours in $A^{> x}$ and $A'^{> x}$ could only possibly differ on $F(b)$, but as $A^{> a} \equiv_1 A^{> b}$, and $b \in A^{> a}$, there is a point $> a$ 
coloured $F(b)$, and therefore also a point $> b$ (and hence $> x$) coloured $F(b)$. A similar argument applies if $x > b$ (using $A^{< a} \equiv_1 A^{< b}$). Finally, if $x = b$ then we can see that 
$A^{< b} \equiv_1 A^{< a} = A'^{< a}$ and $A^{> b} \equiv_1 A^{> a} \equiv_1 A'^{> a}$ (since these last two exhibit the same colours).   \end{proof}

\section{3-equivalence classes}

To help analyze the behaviour of strings up to 3-equivalence, we introduce various labelled directed graphs to keep track of the transitions between 2-characters as we pass through the string. The basic 
idea is that if $a_1a_2a_3 \ldots a_k$ is a string over an alphabet of $m$ colours, then a node of the digraph will be taken to be a 2-character of the form $\langle x, y \rangle_c$ and we include a
directed edge from $\langle x_1, y_1 \rangle_{c_1}$ to $\langle x_2, y_2 \rangle_{c_2}$ provided that for some strings $u_1, v_1, u_2, v_2$, $x_1 = [u_1]_2$, $x_2 = [u_2]_2$, $y_1 = [v_1]_2$, $y_2 = [v_2]_2$
and $u_2 = u_1c_1$, $v_1 = c_2v_2$. This corresponds to the fact that the string $a_1a_2a_3 \ldots a_k$ gives rise to a path 

\vspace{.1in}

$\langle[\emptyset]_2, [a_2\ldots a_k]_2\rangle_{F(a_1)} \longrightarrow \langle[a_1]_2, [a_3\ldots a_k]_2\rangle_{F(a_2)} \longrightarrow $ 

$\langle[a_1a_2]_2, [a_4\ldots a_k]_2\rangle_{F(a_3)} \longrightarrow \ldots  \longrightarrow \langle[a_1a_2a_3 \ldots a_{k-1}]_2, [\emptyset]_2\rangle_{F(a_k)}.$

\vspace{.1in}

In practice, retaining all of both co-ordinates is too cumbersome, and we use an abbreviated string which at any rate for points in the `middle', suffices to describe the 2-character. The object here is to 
obtain some lower bounds on $g(m, 3)$ in the notation of \cite{mwesigye2}, by producing as long optimal strings as possible. The easiest way in which optimality can be assured is to arrange that all points 
have distinct 2-characters. That this is not {\em necessary} for optimality is later remarked (by contrast with Theorem \ref{2.5} for $\equiv_2$).

To illustrate this, we show how for $m = 2$ we can construct an optimal string $A$ of length 70. Certain features of this seem to depend heavily on the specifics of this case, and we are unsure of how to 
generalize. The string is composed of three sections, $L, M$, and $R$ (of lengths 19, 32, and 19) with $L < M < R$. The idea is that, by the time we get to the middle section $M$, the 2-character has 
sufficiently `settled down', to enable us to handle substrings rather uniformly. To describe what $L$, $M$, and $R$ are in this case, we emphasize that this subdivision applies just to the 3-move case. The
subdivision used in section 2 for the 2-move case is also used, but on the left and right subsets, where inductively, and using Theorem \ref{1.1} we need to look at 2-characters. To avoid confusion, we
use $L$, $M$, and $R$ to stand only for the subdivision of the whole string, and if we need to refer to the subdivision of a left or right segment, we use the terms `left', `middle', and `right'. The 
definition here is that $M$ comprises all those points whose left and right 2-characters both themselves have non-empty middle sections. Since it is clear that $M$ so defined is convex, we can then let
$L$ and $R$ be the subsets of its complement which are to its left, right respectively.

Since we shall take $L = rrrrrrbbbbbbrbbbbbr$, the discussion given in the previous section shows that for any $a$ in $M$, 
$[A^{<a}]_2$ begins $rrb$, and it ends with $rb$, $rbb$, $br$, or $brr$ (since we must have $x_1 < x_2 < y_2 < y_1$), and as $(x_2, y_2)$ contains points of both colours, the middle may be taken as $rb$. This means that we can 
essentially describe the left 2-character of a point $a$ in $M$ by the ending of $A^{<a}$ (and its colour). Although the ending will actually have length 2 or 3, we can tell what it is just from its 
last two points. Taking $[A^{>a}]_2$ into account in a similar way, a point is entirely characterized by just 5 entries, two on the left, two on the right, and the colour of $a$ in the middle. The 
following general lemma is invoked here just for $m = 2$ and $k = 5$, but may be more widely applicable.  

\begin{lemma} If $m, k$ are integers $\ge 2$, then there is a string of length $m^k$ with entries in $\{0, 1, \ldots, m-1\}$ such that every string of length $k$ in $\{0, 1, \ldots, m-1\}$ arises exactly 
once as a substring of $k$ consecutive entries (counting cyclically).    \label{3.1}
\end{lemma}

\begin{proof} This method, using an eulerian circuit, was pointed out by P J Cameron.

We form a digraph having as vertices all strings over $\{0, 1, \ldots, m-1\}$ of length $k-1$, and including a directed edge from $(x_1, x_2, \ldots, x_{k-1})$ to $(y_1, y_2, \ldots, y_{k-1})$ provided that 
$(x_2, x_3, \ldots, x_{k-1}) = (y_1, y_2, \ldots, y_{k-2})$. Then the directed edges starting at $(x_1, x_2, \ldots, x_{k-1})$ are those of the form $(x_1, x_2, \ldots, x_{k-1}, y)$ so there are exactly
$m$ of them, and similarly the in-degree of each vertex is also $m$. Furthermore, the digraph is strongly connected, since there is a path from $(x_1, x_2, \ldots, x_{k-1})$ to $(y_1, y_2, \ldots, y_{k-1})$
passing by way of $(x_2, x_3, \ldots, x_{k-1}, y_1)$, $(x_3, x_4, \ldots, x_{k-1}, y_1, y_2)$, $\ldots$, $(x_{k-1}, y_1, y_2, \ldots, y_{k-2})$. Hence the digraph has an eulerian circuit, and this provides 
the desired string of length $m^k$.
\end{proof}

\vspace{.1in}

Given this lemma, we can form a binary string of length 32, such that cyclically ordered, every 5-element string arises exactly once, and this may be taken explicitly as
$$M = rbrbrrbbbrbrbbrbbbbbrrrrrbrrrbbr.$$
To form our sequence of length 70, we precede $M$ by $L = rrrrrrbbbbbbrbbbbbr$ and succeed it by $R = rbbbbbrbbbbbbrrrrrr$ (which is $L$ in reverse, easing some verifications). Let us write this string as $a_1a_2a_3 \ldots a_{70}$. We verify that all 70 points have distinct 2-characters.

First we can see that for every point $a$ of $M \cup R$, in $A^{<a}$, $x_1 = a_1$, $x_2 = a_7$, $y_1 = a_i$ and $y_2 = a_j$ where $i \ge 19$ and $i > j \ge 18$, so its left 2-character begins $rrbrb$, and
ends $rb$, $rbb$, $br$, or $brr$. However, if $a \in L$, $A^{<a}$ has the form $r^ib^jrb^k$, $r^ib^jr$, $r^ib^j$, or $r^i$ for some $i, j, k$, so its left 2-character is {\em not} of this form. By symmetry, 
we can see that no point of $L \cup M$ shares a right 2-character with a point of $R$. We now treat each of $L$ and $M$ individually (and $R$ is similar to $L$).

For $L$, the left 2-characters of the 19 points are $\emptyset$, $r$, $rr$, $rrr$, $rrr$, $rrr$, $rrr$, $rrrb$, $rrrbb$, $rrrbbb$, $rrrbbb$, $rrrbbb$, $rrrbbb$, $rrbbbr$, $rrbbrb$, $rrbbrbb$, $rrbbrbb$, 
$rrbbrbb$, $rrbbrbb$ respectively. The fourth point for each of the repeated left 2-characters has a different colour from the others, and the three remaining points are distinguished by their right 
2-characters, which in each case are distinct members of $\{rrbrbbrr$, $rbrbbrr$, $bbrrbbrr$, $brrbbrr\}$.

Finally, we can see that all points of $M$ have distinct 2-characters since they are midpoints of distinct 5-element strings---notice that we have arranged things so that 
$a_{18}a_{19}a_{20}a_{21} = a_{50}a_{51}a_{52}a_{53}$, which means that distinctness of the 5-element strings persists even at the `ends'.

This shows that $g(2,3)$, which is defined in \cite{mwesigye2} to be the maximum of the lengths of optimal representatives of finite 2-coloured strings under $\equiv_3$ is at least 70. The upper bound given 
in \cite{mwesigye2} is clearly absurdly high, but even so, 70 is a big increase on the optimal length for $m = n = 2$ which is 8. 

Let us see that this is the best we can do by these methods, in which optimality is guaranteed by distinctness of the 2-characters. We can always subdivide a given 2-coloured finite linear order into 3 
sections, $L$, $M$, and $R$, where in $M$, both left and right 2-characters have $rb$ as `middle'. Clearly $M$ is convex, so we may take $L$ and $R$ to be the sets of points to the left, right of $M$ 
respectively. If $a \in M$, then the left and right 2-characters of $a$ must have at least 6 entries, and as in the discussion above, the 2-characters of the points of $M$ are entirely determined by the 
5-element strings of which they are mid-points. Hence $|M| \le 32$. 

Now consider what $L$ can be. Without loss of generality, suppose it begins with $r$. If it has an initial segment of the form $r^ib^jr^kb^lr^pb^q$ with positive exponents, then the next point does not lie 
in $L$, and similarly, $q = 1$ (since otherwise the final point does not lie in $L$), and by similar arguments, $j = l = 1$, giving $L = r^ibr^kbr^pb$. If $i \ge 7$ then the fourth and fifth points of $A$ 
realize the same 2-characters, contrary to assumption. Hence $i \le 6$. Similarly, $k, p \le 6$. If $k = 6$, then 2-characters are repeated for the middle two entries in that block. Hence $k \le 5$, and 
similarly $p \le 5$. Hence $|L| \le 6 + 5 + 5 + 3 = 19$ (and one can check that $rrrrrrbrrrrrbrrrrrb$ is possible).

If $L = r^ib^jr^kb^lr^p$ it again follows that $j = l = 1$, $i \le 6$, $k, p \le 5$, so $|L| \le 18$.

If $L = r^ib^jr^kb^l$ then $j = 1$ or $k = 1$, and again, $i \le 6$, $l \le 5$, and also $j, k \le 5$, so $|L| \le 6 + 1 + 5 + 5 = 17$.

If $L = r^ib^jr^k$ then $i \le 6$, $j \le 7$, $k \le 6$ so $|L| \le 19$.

If $L = r^ib^j$ then $|L| \le 14$ and if $L = r^i$ then $|L| \le 7$. 

It follows similarly that $|R| \le 19$, and hence $|A| \le 19 + 32 + 19 = 70$.

Finally we remark that in $\equiv_3$-optimal strings, 2-characters may be repeated, and using this we are able to construct a longer $\equiv_3$-optimal 2-coloured string. We first give a small example. 
Consider $A = rbrbrbrbrbrbrbr$, which has length 15, is a palindrome (reading the same forwards and backwards), and whose 7th and 9th entries realize that same 2-character (though apart from this, all 
2-characters are distinct). To see that $A$ is $\equiv_3$-optimal, suppose that $B \equiv_3 A$, and we show that $B$ has length at least 15. S ince $A$ realizes 14 2-characters, so does $B$, and hence it 
has length at least 14. Now $A$ realizes the 2-character $\langle rbrb, brrbbr \rangle_r$, so $B$ must realize this as well, and as $rbrb$ lies in a singleton $\equiv_2$-class, $B$ begins $rbrbr$. 
Similarly, $B$ realizes $\langle rbrbr, rbrbbr \rangle_b$, so as the $\equiv_2$-class of $rbrbr$ is $\{rbr^pbr: p \ge 1\}$, $B$ begins with $rbr^pbrb$ for some $p \ge 1$. Since $B$ begins with $rbrbr$ it 
follows that $p = 1$, and that $B$ begins $rbrbrb$. Similarly $B$ ends with $brbrbr$. The other two 2-characters realized by $B$ are $\chi_1 = \langle rbrbrb, brrbbr \rangle_r$ and 
$\chi_2 = \langle rbrbbr, rbrbbr \rangle_b$. Since $B^{<b_7} = rbrbrb$, $b_7$ must realize $\chi_1$, and so $b_7 = r$. Similarly, the 7th point from the right realizes $\chi_1$ and is $r$. Since 
$|B| \ge 14$, these two points are distinct, and as $B$ also realizes $\chi_2$, there must be another point between them, so $B$ has length at least 15.

We now present a 3-optimal string of length 74 having the same $L$ and $R$ as in the example given of length 70, but with longer $M$:
$$A = rrrrrrbbbbbbrbbbbbr|rbrbbbbbrbbrrbrbrrrrrbrrrbbrbrrrbbbr|rbbbbbrbbbbbbrrrrrr.$$
The subdivision into $L$, $M$, $R$ is indicated. By the previous discussion, this must have repeated 2-characters in $M$. We let $L' = br$ (the end of $L$) and $R' = rb$ (the beginning of $R$). To verify 
that $A$ is $\equiv_3$-optimal we first note that by the previous arguments, the 2-characters of all elements of $L \cup R$ are uniquely determined and different from all those occurring in $M$, so any 
3-equivalent string must begin and end in this way. Now $A$ determines a path through the digraph $D$ having as vertices strings of length 4 over $\{r, b\}$ arising as convex subsets of $L' \cup M \cup R'$,
and the edges of $D$ tell us precisely which 2-characters are realized in the path. So although some will now be repeated, we have to check that no shorter path through $D$ can realize precisely the same 
2-characters. One checks that $D$ has 28 edges. Suppose therefore that $P$ is a path through $D$ traversing precisely the same edges (though not necessarily the same number of times). We shall show that
$P$ has length at least 36. We may also view $P$ as a `multi-digraph' in which the multiplicities with which the edges of $D$ arise in $P$ are also recognized, and in this sense we may talk of the 
`in-degree' $in(x)$ and `out-degree' $out(x)$ of a vertex $x$. By the usual theory of eulerian paths, if $i$ and $f$ are the initial and final vertices of $P$, then $i \neq f$ implies that 
$out(i) = in(i) + 1$ and $in(f) = out(f) + 1$; all other vertices $x$, and also $i$ and $f$ if they are equal, satisfy $in(x) = out(x)$. Furthermore, since $L' = br$, $i$ must equal $brrr$, $brrb$, $brbr$, 
or $brbb$, and similarly $f$ must equal $rrrb$, $rbrb$, $brrb$, or $bbrb$.The digraph $D$ is shown in figure 1.

\begin{figure}
\setlength{\unitlength}{1mm}
\begin{picture}(00,00)(65,125)

\curve(6,15,3,17,1.5,13)
\curve(6,11,3,9,1.5,13)
\put(1.5,13.2){\vector(0,-1){1}}

\put(10,13){\makebox(0,0)[cc]{$rrrr$}}
\put(14,13){\vector(3,1){6}}

\put(24,9){\makebox(0,0)[cc]{$brrr$}}
\put(20,9){\vector(-3,1){6}}
\put(24,10.5){\vector(0,1){4}}

\put(24,17){\makebox(0,0)[cc]{$rrrb$}}
\put(28,18){\vector(3,1){6}}
\curve(28.2,16.4,38,15,47.8,13.6)
\put(47.45,13.68){\vector(4,-1){1}}

\put(38,5){\makebox(0,0)[cc]{$rbrr$}}
\put(34,5){\vector(-3,1){7}}

\put(38,21){\makebox(0,0)[cc]{$rrbr$}}
\put(38,18.5){\vector(0,-1){11.5}}
\curve(42.2,20.4,52,19,61.8,17.6)
\put(61.1,17.72){\vector(4,-1){1}}

\put(52,13){\makebox(0,0)[cc]{$rrbb$}}
\curve(56.2,13.3,80,15,103.8,16.7)
\put(103.1,16.65){\vector(1,0){1}}
\put(54.3,11){\vector(1,-1){8.5}}

\put(66,1){\makebox(0,0)[cc]{$rbbr$}}
\curve(70.2,1.6,80,3,89.8,4.4)
\put(89.8,4.45){\vector(1,0){1}}
\put(68,2.5){\vector(1,1){8.5}}

\put(66,9){\makebox(0,0)[cc]{$brbr$}}
\curve(41.85,5.55,52,7,62.15,8.45)
\put(42.55,5.67){\vector(-4,-1){1}}

\put(66,17){\makebox(0,0)[cc]{$rbrb$}}
\curve(70.2,17.6,80,19,89.8,20.4)
\put(89.1,20.33){\vector(1,0){1}}
\put(66,15){\vector(0,-1){4}}

\put(66,25){\makebox(0,0)[cc]{$brrb$}}
\curve(42.2,21.6,52,23,61.8,24.4)
\put(42.55,21.65){\vector(-4,-1){1}}

\put(80,13){\makebox(0,0)[cc]{$bbrr$}}
\put(78,14.5){\vector(-1,1){8.3}}

\put(94,5){\makebox(0,0)[cc]{$bbrb$}}
\put(94,7.5){\vector(0,1){10.5}}
\curve(70.2,8.4,80,7,89.8,5.6)
\put(70.55,8.35){\vector(-1,0){1}}

\put(94,21){\makebox(0,0)[cc]{$brbb$}}
\curve(93,18,82,10,69,2)
\put(69.7,2.5){\vector(-4,-3){1}}
\put(98,20){\vector(3,-1){7}}

\put(108,9){\makebox(0,0)[cc]{$bbbr$}}
\put(104,8.1){\vector(-3,-1){7}}
\curve(84.2,12.4,94,11,103.8,9.6)
\put(84.85,12.6){\vector(-1,0){1}}

\put(108,17){\makebox(0,0)[cc]{$rbbb$}}
\put(108,15){\vector(0,-1){4}}
\put(112,16){\vector(3,-1){7}}

\put(122,13){\makebox(0,0)[cc]{$bbbb$}}
\put(118,12){\vector(-3,-1){6.8}}

\curve(126,15,129,17,131,13)
\curve(126,11,129,9,131,13)
\put(131,13.2){\vector(0,-1){1}}

\end{picture}
\end{figure}

\vspace{1.5in}
\qquad\qquad\qquad\qquad\qquad\qquad\qquad\qquad Figure: 1

\vspace{.2in}

Now we note that the vertex $rrbb$ is therefore an internal vertex of $P$, so has equal in- and out-degrees in $P$. Since its out-neighbours (in $P$) $rbbr$ and $rbbb$ are distinct, it follows that its 
in-degree is at least 2. Similarly, $in(rbrb)$, $out(rbrr)$, $out(brbr)$, $out(bbrr) \ge 2$. In $D$, each of $rrbb$ and $rbrb$ has only one in-neighbour, so the corresponding edges in $P$ must each appear 
at least twice. Similarly for the out-neighbours of $rbrr$, $brbr$, and $bbrr$. This already assures us that $P$ has length at least 33. But now we know that $in(rbrr) \ge 3$, and as $rbrr$ is internal,
also $out(rbrr) \ge 3$. Similarly, $in(rrbr)$, $in(rrrb) \ge 3$. Since the extra edges thereby assured and contributing to $out(rbrr)$, $in(rrbr)$, and $in(rrrb)$ must be $rbrr \to brrr$, 
$rrrb \to rrbr$ or $brrb \to rrbr$, and $rrrr \to rrrb$ or $brrr \to rrrb$, this gives at least 3 extra edges in $P$, showing that it has length at least 36, as desired.

\vspace{.1in}

\noindent{\bf Future work}

We have really only scratched the surface of this topic, in \cite{mwesigye2} and here, and a great deal more effort would be required to understand fully the structure of $\equiv_n$-optimal strings for all 
$n$, and for all colour set sizes. The method in the final example just given seems very laborious, merely to increase the length by 4, and that is only for two colours and $n = 3$. Undoubtedly there will 
be longer examples, requiring more careful checking. We remark that we also believe that one can construct finite strings having no $\equiv_3$-equivalent optimal substring. The idea would be to find a 
string as above obtained by modifying the length 70 example, but such that this time, the path is {\em not} optimal, but that any optimal path traversing the same edges as $D$ would have to have them in a 
different order, so that the optimal string wouldn't actually be a substring of the original one. 

We conclude by illustrating the specific problem which applies even to increase the number of colours by 1. We would like to apply the same kind of analysis as for the case of 2 colours, which relied on the 
subdivision of the string $A$ into $L$, $M$, $R$. We can still do this, and for long enough strings there will be a non-trivial $M$, comprising those points such that the left and right 2-characters both 
themselves have a `middle' in which all 3 colours, red, blue, and green, appear. Since this time 2-characters have length up to 15 (see \cite{mwesigye2}) the lengths of $L$ and $R$ will usually be a lot 
longer. The main problem comes about in $M$ however. Last time we were able to pin down the 2-character of a point in $M$ by a sequence of length 5. This time though, the right end of the left 2-character 
may be $rbg$, $rbgg$, $rbbgg$, $rbgbg$, $rbggb$, as well as others obtained by permuting the three colours, and we seem to need the final 4 entries at least to tell which is which, and so we'd have to have 
sequences of length 7, 8, or 9, in place of 5, that is, not constant. Thus the trick of using an Euler circuit doesn't seem to apply.

%%%%%%%%%%%%%%%%%%%%%%%%%%%%%%%%%%%%
%% Bibliography
%%%%%%%%%%%%%%%%%%%%%%%%%%%%%

\end{document}